\documentclass[12pt,reqno]{amsart}

\usepackage{graphicx,subfigure}
\usepackage{hyperref}
\hypersetup{colorlinks=true, citecolor=blue, linkcolor=red}

\numberwithin{equation}{section} \numberwithin{figure}{section}
\numberwithin{table}{section} 
{ \theoremstyle{plain}

}

{ \theoremstyle{definition}
\newtheorem{thm}{Theorem}

\newtheorem{rem}{Remark}

\numberwithin{equation}{section} \numberwithin{lem}{section}
\numberwithin{thm}{section} \numberwithin{cor}{section}
\numberwithin{pro}{section} \numberwithin{rem}{section}
}

\begin{document}

\title[A Liouville type result for a class of  variational  elliptic systems]{A Liouville type result for bounded, entire solutions to
a class of variational semilinear elliptic systems}



\author{Christos Sourdis} \address{Department of Mathematics and Applied Mathematics, University of
Crete.}
              \email{csourdis@tem.uoc.gr}           




\subjclass{Primary: 35J20, 35J91; Secondary: 35J47.}

\keywords{elliptic systems, Liouville theorem, Ginzburg-Landau
equation, monotonicity formula, multi-phase transitions}

\maketitle

\begin{abstract}
We prove a Liouville type result for bounded, entire solutions to
a class of variational semilinear elliptic systems, based on the
growth of their potential energy over balls with growing radius.
Important special  cases to which our result applies are  the
Ginzburg-Landau system  and systems that arise in the study of
multi-phase transitions.
\end{abstract}

We consider the semilinear elliptic system
\begin{equation}\label{eqEq}
\Delta u= W_u(u)\ \ \textrm{in}\ \ \mathbb{R}^n,\ \ n\geq 2,
\end{equation}
where  $W:\mathbb{R}^m\to \mathbb{R}$, $m\geq 2$, is sufficiently
smooth and
\[
W\geq 0.
\]
Under these general assumptions, it is well known (see for
instance \cite{alikakosBasicFacts,bethuelBIG}) that solutions of
(\ref{eqEq}) satisfy the following monotonicity property:
\begin{equation}\label{eqmonotoniWeak}
\frac{d}{dr}\left(\frac{1}{r^{n-2}}\int_{B_r}^{}\left\{\frac{1}{2}|\nabla
u|^2+ W\left(u\right) \right\}dx\right)\geq 0,\ \ r>0,
\end{equation}
where $B_r$ stands for the $n$-dimensional ball of radius $r$ and
center at the origin (keep in mind that (\ref{eqEq}) is
translation invariant). As a direct consequence, we have the
following Liouville type theorem:
\\
\textbf{Liouville property A} \emph{The only solutions of
(\ref{eqEq}) with finite energy, that is
\begin{equation}\label{eqenerWhole}
\int_{\mathbb{R}^n}^{}\left\{\frac{1}{2}|\nabla u|^2+
W\left(u\right) \right\}dx<\infty,
\end{equation}
are the constants ones.}
\\
(The case $n=2$ requires some extra work, see
\cite{alikakosBasicFacts}).

 If $m=1$, bounded solutions of (\ref{eqEq}) satisfy
the gradient bound:
\begin{equation}\label{eqmodica}
\frac{1}{2}|\nabla u|^2\leq W(u) \ \ \textrm{in}\ \ \mathbb{R}^n,
\end{equation}
(see\ \cite{cafamodica,farinaFlat,modica}), and thus a stronger
Liouville property holds:
\\
\textbf{Liouville propertly B} \emph{The only bounded solutions of
(\ref{eqEq}) with finite potential energy, that is
\begin{equation}\label{eqpotFin}\int_{\mathbb{R}^n}^{} W\left(u\right) dx<\infty,
\end{equation}
are the constant ones.}
\\
In fact, a stronger monotonicity formula holds  with $n-1$ in
place of $n-2$ in (\ref{eqmonotoniWeak}) (see
\cite{cafamodica,modicaProc}).

If $m\geq 2$, the analog of the gradient bound (\ref{eqmodica})
does not hold in general, as is indicated by the counterexamples
in \cite{farinatwores,smyrnelis}.

Despite of this, there are such systems where the Liouville
property B is in effect. Indeed, let us consider the systems,
arising in the study of multi-phase transitions (see
\cite{baldo}), where $W$ has a finite number of global minima
which, in addition, are non-degenerate. Let $u$ be a bounded
solution of (\ref{eqEq}) such that (\ref{eqpotFin}) holds. Using
that $|\nabla u|$ is also bounded (by standard elliptic estimates
\cite{Gilbarg-Trudinger}) in the same way as in
\cite{brezisMerle}, and exploiting that the global minimizers of
$W$ are isolated, we find that there exists a global minimizer of
$W$, say $a\in \mathbb{R}^m$, such that
\begin{equation}\label{eqphase}
\left|u(x)-a \right|\to 0\ \ \textrm{as}\ \ |x|\to \infty.
\end{equation}
In turn, the non-degeneracy of $a$ yields that
\[
\left|u(x)-a \right|+|\nabla u|\leq C_0e^{-C_1|x|}, \ \ x\in
\mathbb{R}^n,
\]
for some constants $C_0,C_1>0$ (see \cite{gui}). The above
relation clearly implies that (\ref{eqenerWhole}) holds and,
thanks to the Liouville property A, we conclude that $u\equiv a$.

In the case where the global minima of $W$ are not isolated, the
previous simple argument fails. Nevertheless, it was shown in
\cite{farinatwores} that the Liouville property B holds for the
Ginzburg-Landau system
\begin{equation}\label{eqGL}
\Delta u=\left(|u|^2\!-\!1 \right)u,   \   (\textrm{here}\
W(u)=\frac{\left(1\!-\!|u|^2\right)^2}{4}\ \textrm{vanishes on}\
\mathbb{S}^{m\!-\!1}  ),
\end{equation}
arising in superconductivity, provided that $n\geq 4$. On the
other hand, this property does not hold when $n=2$ (see
\cite{brezisMerle}), whereas it was shown to hold if $n=3$ and
$m=2$ in \cite{farinaWorld}. The approach in \cite{farinatwores}
is based on standard elliptic estimates and the monotonicity
formula (\ref{eqmonotoniWeak}). On the other side, let us mention
that for this system one can control the growth of the potential
energy (over growing balls) in terms of the corresponding growth
of the kinetic energy (see \cite{brezisMerle,farinaDiffEqs}).

In this note, we will prove that a stronger assertion  than the
Liouville property B holds for a broad class of systems, which
includes both the phase transition and the Ginzburg-Landau system
that we mentioned above.   As will be apparent, another advantage
of our proof is its simplicity.


Our main result is the following.


\begin{thm}\label{thm1}
Assume that $W\in C^1(\mathbb{R}^m;\mathbb{R})$ is nonnegative and
satisfies
\begin{equation}\label{eqASS}
-(u-Q)\cdot  W_u(u)\leq C_2 \left(W(u)\right)^{\frac{p-1}{p}},\ \
u\in \mathbb{R}^m,
\end{equation}
for some $Q\in \mathbb{R}^m$, $p\geq 2$ and a constant $C_2>0$
(here $\cdot$ stands for the Euclidean inner product of
$\mathbb{R}^m$). Let $u \in C^2(\mathbb{R}^n;\mathbb{R}^m)$ be a
bounded solution of (\ref{eqEq}).
\begin{description}
    \item[I] If $n\geq 4$, the following Liouville property holds:
\emph{\begin{equation}\label{eqpotTHM}
\int_{B_r}^{}W(u)dx=o(r^{q})\ \ \textrm{as}\ \ r\to \infty,
\end{equation}
with
\begin{equation}\label{eqQ}
q=\frac{p}{p-1}\left(n-2-\frac{n}{p} \right)=n-2-\frac{2}{p-1},
\end{equation}
 implies that $u$ is a constant}, (where $o(\cdot)$ is the
standard Landau symbol).

    \item[II]If $n=4$ and $p=2$, the Liouville property B holds, that is (\ref{eqpotFin})
implies that $u$ is a constant.
\end{description}
\end{thm}
\begin{proof}
We consider first the Case I. Without loss of generality, we may
assume that $Q=0$. The main effort will be placed in estimating
the corresponding growth of the kinetic energy. To this end, we
will adapt an argument from \cite{weiWeth}. We take the inner
product of (\ref{eqEq}) with $u$ and integrate by parts over
$B_r$, $r>0$, to arrive at
\begin{equation}\label{eqBas}
\int_{B_r}^{}|\nabla u|^2 dx=-\int_{B_r}^{} u\cdot  W_u(u) dx+
\int_{\partial B_r}^{}u\cdot u_r dS,
\end{equation}
where $u_r$ denotes the radial derivative of $u$. The first term
in the righthand side will be estimated initially  by
(\ref{eqASS}). Concerning the second term, letting
\[
I(r)=\frac{1}{r^{n-1}}\int_{\partial B_r}^{}|u|^2 dS,
\]
we note that
\[
I'(r)=\frac{2}{r^{n-1}}\int_{\partial B_r}^{}u\cdot u_r dS,
\]
(see also \cite[Ch. 6]{colding}). Now, from equation
(\ref{eqBas}),  we obtain
\begin{equation}\label{eqBas2}\begin{array}{rcl}
  \int_{B_r}^{}|\nabla u|^2 dx & \leq & C_2\int_{B_r}^{}  \left(W(u)\right)^{\frac{p-1}{p}} dx+\frac{1}{2}r^{n-1}I'(r) \\
    &   &   \\
    & \leq & C_3 r^{\frac{n}{p}} \left(\int_{B_r}^{} W(u)dx \right)^{\frac{p-1}{p}}+\frac{1}{2}r^{n-1}I'(r) \\
   &   &   \\
\textrm{using}\ (\ref{eqpotTHM}) & \leq  & o(r^{n-2})+\frac{1}{2}r^{n-1}I'(r), \\
\end{array}
\end{equation}
for some constant $C_3>0$, as $r\to \infty$.

Next, we claim that there exists a sequence $r_j\to \infty$ such
that
\[
r_j I'(r_j)\to 0.
\]
If not, without loss of generality, there would exist  positive
constants $\delta$, $R_0$ such that
\[
I'(r)\geq \frac{\delta}{r},\ \ r\geq R_0,
\]
which implies that
\[
I(r)\geq I(R_0)+\delta\ln\left(\frac{r}{R_0}\right),\ \ r\geq R_0.
\]
However, this is not possible  since $I$ is a bounded function
(from the assumption that $u$ is bounded).

To conclude, we take $r=r_j$ in (\ref{eqBas2}) to find that
\[
\int_{B_{r_j}}^{}|\nabla u|^2 dx\leq o(r_j^{n-2})\ \ \textrm{as}\
j\to \infty.
\]
Clearly, the monotonicity formula (\ref{eqmonotoniWeak}), the
assumption (\ref{eqpotTHM}) and the above relation yield that $u$
is a constant, as desired.

The proof in the Case II requires some minor modifications. Again
we may assume that $Q=0$. Motivated from \cite{farinatwores}, we
now integrate the relation $u\cdot \Delta u= u \cdot W_u(u)$ over
the annulus  $B_t \setminus B_s$ with $t>s$. Then, instead of
(\ref{eqBas2}), we  have
\[
\int_{B_t\setminus B_s}^{}|\nabla u|^2dx\leq C_3 t^2
\left(\int_{\mathbb{R}^n\setminus
B_s}^{}W(u)dx\right)^\frac{1}{2}+O(s^{3})+\frac{1}{2}t^{3}I'(t),
\]
for all $t>s>0$, where we also used that $|\nabla u|$ is bounded
in $\mathbb{R}^n$ (by standard elliptic estimates
\cite{Gilbarg-Trudinger}). In turn, this implies that
\[
\int_{B_t}^{}|\nabla u|^2dx\leq C_3 t^2
\left(\int_{\mathbb{R}^n\setminus
B_s}^{}W(u)dx\right)^\frac{1}{2}+O(s^{4})+\frac{1}{2}t^{3}I'(t),
\]
for all $t>s>1$. As before, for any $s>1$, there exists a sequence
$\{t_j \}$ such that $t_j\to \infty$ and
\begin{equation}\label{eqvia}
\int_{B_{t_j}}^{}|\nabla u|^2dx\leq C_3 t_j^2
\left(\int_{\mathbb{R}^n\setminus
B_s}^{}W(u)dx\right)^\frac{1}{2}+O(s^{4})+o(t_j^2),
\end{equation}
as $j\to \infty$. Let $r>0$ (independent of $s,j$). By the
monotonicity formula (\ref{eqmonotoniWeak}), we obviously have
that
\[
\frac{1}{r^{2}}\int_{B_r}^{}\left\{\frac{1}{2}|\nabla u|^2+
W\left(u\right) \right\}dx \leq
\frac{1}{t_j^{2}}\int_{B_{t_j}}^{}\left\{\frac{1}{2}|\nabla u|^2+
W\left(u\right) \right\}dx,
\]
provided that $j$ is sufficiently large (depending on $r,s$).
Letting $j\to \infty$, via (\ref{eqvia}), we obtain that
\[
\frac{1}{r^{2}}\int_{B_r}^{}\left\{\frac{1}{2}|\nabla u|^2+
W\left(u\right) \right\}dx\leq C_3
\left(\int_{\mathbb{R}^n\setminus
B_s}^{}W(u)dx\right)^\frac{1}{2},
\]
for all $r>0$ and $s>1$. Finally, letting $s\to \infty$ and
recalling (\ref{eqpotFin}), we deduce  that $u$ is a constant.

 The proof of the theorem is complete.
\end{proof}

\begin{rem}
Condition (\ref{eqASS}) holds, with $p=2$ (and for any $Q\in
\mathbb{R}^m$), if $W\geq 0$ is $C^2$ and vanishes on finitely
many smooth, compact manifolds which are non-degenerate, in the
sense that at each point $u$ of such a manifold $M$ we have $T_u
M=\textrm{Ker}\left[W_{uu}(u) \right]$ (where $T_u M$ denotes the
tangent space of $M$ at the point $u$).
\end{rem}

\begin{rem}
Let us revisit the  phase transition systems that we mentioned
previously in relation to (\ref{eqphase}). Condition (\ref{eqASS})
with $p=2$ is clearly satisfied near the global minima of $W$ (for
any $Q\in \mathbb{R}^m$), and thus also in bounded subsets of
$\mathbb{R}^m$. In fact, it is easy to see that the same is true
(with $p\geq 2$) in the more general case where near each zero $a$
of $W$ the following holds: There exist $c>0$ and $p\geq 2$ such
that
\begin{equation}\label{eqdeg}
W(u) \geq c |u-a|^p\ \ \textrm{and}\ \ \left|W_u(u) \right|\leq
\frac{1}{c} |u-a|^{p-1}
\end{equation}
(this is certainly satisfied if $W$ vanishes up to a finite order
at $a$). In passing, let us note that this class of systems
appears in the recent study \cite{smets}. For such $W$, we showed
very recently in \cite{sourdisCPAA} that, for any $k>0$, bounded,
nonconstant solutions of (\ref{eqEq}) satisfy:
\begin{equation}\label{eqcpaa}
\frac{1}{(\ln r)^k}\frac{1}{
r^{n-2}}\int_{B_r}^{}\left\{\frac{1}{2}|\nabla u|^2+
W\left(u\right) \right\}dx\to \infty\ \ \textrm{as}\ \ r\to
\infty,
\end{equation}
(actually, this was shown in the case of nondegenerate minima but
the proof carries over straightforwardly to this degenerate case).
Armed with the above information, and modifying the proof of
Theorem \ref{thm1} accordingly, we obtain the following Liouville
property for bounded solutions to such systems in $n\geq 4$
dimensions:
\[
\int_{B_r}^{}W(u)dx=O\left((\ln r)^k r^{q}\right), \ \ \textrm{for
some}\ k>0,\ \ \textrm{as}\ \ r\to \infty \ \Longrightarrow\ u\
\textrm{is a constant},
\]
where the exponent $q$ is as in (\ref{eqQ}).

We cannot resist to make the following observation. If $W$ has
only one such global minimum, say at the origin, it holds that
\[
\left| u\cdot W_u(u) \right| \leq C_4 W(u)
\]
for some constant $C_4>0$  (at least in  the range of the
considered bounded solution). Therefore, using this to estimate
directly the first term in the righthand side of (\ref{eqBas}),
and modifying slightly the rest of the proof, we infer that a
stronger Liouville property holds for any solution in $n\geq 2$
dimensions:
\begin{equation}\label{eqConj}
\int_{B_r}^{}W(u)dx=o\left(r^{n-2}\right)\ \ \textrm{as}\ \ r\to
\infty \ \Longrightarrow\ u\ \textrm{is a constant}.
\end{equation}
It is natural to conjecture that this property continues to hold
when there are arbitrary many such minima. In this direction, see
Remark \ref{remMinim} below.
\end{rem}

\begin{rem}
Assume that $W$ is nonnegative, sufficiently smooth
($C^{2,\alpha}$ with $0<\alpha<1$ suffices), and there exists an
$M>0$ such that
\[
u\cdot W_u(u)>0, \ \ |u|>M.
\]
Under these assumptions, it was shown recently in
\cite{smyrnelis}, in the spirit of \cite{cafamodica}, that there
exists a constant $C_5>0$ such that all bounded solutions of
(\ref{eqEq}) satisfy the gradient bound:
\[
|\nabla u|^2\leq C_5 \left(M^2-|u|^2 \right)\ \ \textrm{in}\ \
\mathbb{R}^n.
\]
Assume further that there exists a constant $C_6>0$ such that
\[
W(u)\geq C_6\left(M^2-|u|^2 \right)^2,\ \ |u|\leq M.
\]
The above assumptions are clearly satisfied by the Ginzburg-Landau
potential (with $M=1$). Furthermore, as was observed in
\cite{smyrnelis}, they are also satisfied by a class of symmetric
phase transition potentials.

It follows that there is a constant $C_7>0$ such that
\[
\int_{B_r}|\nabla u|^2dx\leq C_7 r^{\frac{n}{2}}\left(
\int_{B_r}W(u) dx \right)^\frac{1}{2}, \ \ r>0,
\]
(compare with (\ref{eqBas2}) for $p=2$). Now, we can argue by
contradiction and use the monotonicity formula
(\ref{eqmonotoniWeak}) to show the following: If $u$ is a bounded,
nonconstant solution of (\ref{eqEq}) in $n\geq 5$ dimensions,
there exists a constant $C_8>0$ such that
\[
\int_{B_r}^{}W(u)dx\geq C_8 r^{n-4}, \ \ r\geq 1.
\]
Notice that the above lower bound represents  a slight improvement
over Theorem \ref{thm1} for this particular class of potentials.

Suppose that, in addition, the potential $W$ is of phase
transition type and satisfies (\ref{eqdeg}). Then, we can exploit
(\ref{eqcpaa}) to obtain the stronger property: Given $k>0$,
nonconstant, bounded solutions in $n\geq 4$ dimensions satisfy
\[
\frac{1}{(\ln r)^k}\frac{1}{ r^{n-4}}\int_{B_r}^{} W\left(u\right)
dx\to \infty\ \ \textrm{as}\ \ r\to \infty.
\]
\end{rem}

\begin{rem}\label{remMinim}
In the case of phase transition potentials with non-degenerate
global minima, it was shown recently in \cite{AlikakosDensity}
that each nonconstant, bounded and local minimizing solution (in
the sense of Morse)  satisfies
\[
\int_{B_r}^{}\left\{\frac{1}{2}|\nabla u|^2+ W\left(u\right)
\right\}dx\geq C_9 r^{n-1}, \ \ r\geq 1,
\]
for some constant $C_9>0$. Using this instead of the monotonicity
formula in the proof of Theorem \ref{thm1} (with $p=2$), we deduce
that such solutions satisfy the Liouville property (\ref{eqConj}).
Interestingly enough, if $n=2$ and under slightly more general
assumptions, it was shown recently in \cite{sourdis14} that such
solutions satisfy
\[
\int_{B_r}^{}W\left(u\right)dx\geq C_{10} r, \ \ r\geq 1,
\]
for some constant $C_{10}>0$. For related Liouville type results
for minimizing solutions, we refer the interested reader to
\cite{fuscoCPAA}.
\end{rem}

 \section*{Acknowledgements} This research was supported by the
ARISTEIA (Excellence) programme ``Analysis of discrete, kinetic
and continuum models for elastic and viscoelastic response'' of
the Greek Secretariat of Research.

\end{document}